\documentclass[a4paper]{amsart}
\usepackage{amssymb}
\usepackage[size=tiny]{todonotes}
\usepackage{xfrac} 
\usepackage{comment}
\usepackage[
    backend=bibtex,
    style=numeric,
    citestyle=nature,
    maxcitenames=9,
    maxbibnames=9,
    sorting=nyt,
    natbib=true,
    url=false, 
    doi=false,
    eprint=true
]{biblatex}
\newtheorem{theorem}{Theorem}[section] 
\newtheorem{lemma}[theorem]{Lemma}
\newtheorem{proposition}[theorem]{Proposition}
\newtheorem{conjecture}[theorem]{Conjecture}
\newtheorem{corollary}[theorem]{Corollary}

\theoremstyle{definition}

\newtheorem{definition}[theorem]{Definition}
\newtheorem{remark}[theorem]{Remark}

\newcommand{\PKS}{P_{\mathrm KS}}

\newcommand{\RR}{\mathbb{R}}

\newcommand{\ZZ}{\mathbb{Z}}
\newcommand{\NN}{\mathbb{N}}
\newcommand{\PP}{\mathbb{P}}

\title[Diagonalizations of denormalized volume polynomials]{Diagonalizations of\\ denormalized volume polynomials}
\author{Julius Ross}
\author{Hendrik S\"uss}
\address{Mathematics, Statistics and Computer Science, University of Illinois at Chicago, 851 S. Morgan Street, 322 Science and Engineering Offices, Chicago, IL~60607, United States}
\email{juliusro@uic.edu}
\address{Friedrich-Schiller-Universit\"at Jena, Fakult\"at f\"ur Mathematik und Informatik, Institut für Mathematik, Ernst-Abbe-Platz 2, 07743 Jena, Germany}
\email{hendrik.suess@uni-jena.de}

\begin{document}

\begin{abstract}
  We show that diagonalization, products and lower truncations preserve the property of being a denormalized volume polynomial. We also discuss an application to poset inequalities.
\end{abstract}
\maketitle

\section{Introduction}\label{sec:introduction}
Given a  projective variety $X$ of dimension $d$ and ample  divisor classes $D_1,\ldots,D_n$ we call the homogeneous polynomial
\[\sum_{|\alpha|=d} c_{\alpha} \cdot x_1^{\alpha_1}\cdots x_n^{\alpha_n} = d! \sum_{|\alpha|=d} \left(D_1^{\alpha_1} \cdots D_n^{\alpha_n}\right) x_1^{\alpha_1}\cdots x_n^{\alpha_n}\]
of degree $d$ the \emph{associated denormalized volume polynomial}. Here, $\left(D_1^{\alpha_1} \cdots D_n^{\alpha_n}\right)$ denotes the intersection product of the corresponding divisor classes. More generally, we also call  limits of positive multiples of such polynomials \emph{denormalized volume polynomials}. Denormalized volume polynomials are of interest, because they come with strong positivity properties for their coeffiecients. Most prominently the \emph{Khovanskii-Teissier} inequality
\begin{equation}
  \label{eq:khovanskii-teissier}
  c_{\alpha}^2  \geq c_{\alpha-e_i+e_j}\cdot c_{\alpha+e_i-e_j}
\end{equation}
which holds more generally for denormalized Lorentzian polynomials (see \cite{BrandenHuh})  and the \emph{reverse Khovanskii-Teissier} inequality
\begin{equation}
  \label{eq:rKT}
  c_{(\alpha+ \beta+\gamma)} \cdot c_{(\alpha+ |\beta+\gamma|e_1)} \leq \binom{|\beta|+|\gamma|}{|\beta|} c_{(\alpha + |\gamma|e_1+\beta)}\cdot c_{(\alpha + |\beta|e_1+ \gamma)},
\end{equation}
see \cite{zbMATH07003736,zbMATH07745002,hu2023intersection} and Lemma~\ref{lem:rKT} below.

Interesting denormalized volume polynomials arise as generating functions from combinatorial counting problems. As a prominent example the Schur polynomials $s_\lambda$ are denormalized volume polynomials (see \cite{zbMATH07542930}). Their coefficients (called Kostka numbers) count to the number of Young tableaux of shape $\lambda$ for a given weight vector. In this setting positivity properties of the polynomial translate in interesting inequalities for combinatorial quantities.

It is of general interest to determine which operations on polynomials preserve these positivity properties. Consider for example the diagonalization
$$g(x_1,\ldots,x_{n-1}):=f(x_1,\ldots,x_{n-1},x_{n-1})$$ of a denormalized volume polynomial $f(x_1,\ldots,x_n)$. If the positivity properties were preserved we would obtain inequalities of the same type for the sums of coefficients $s_{(\alpha,i)}=\sum_{i+j=d-|\alpha|} c_{(\alpha,i,j)}$. In the case of generating functions for combinatorial quantities these sums might be meaningful quantities, as well. In the case of Kostka numbers we obtain the number of Young tableaux in which only part of the weight vector is fixed.

Our main result states that the property of being a denormalized volume polynomial is indeed preserved by diagonalization.
\begin{theorem}[Diagonalization]
    \label{thm:diagonalization}
  If $$p(v_1,\ldots,v_{k},x_1,\ldots,x_m) \in \RR[v_1, \ldots,v_{k},x_1,\ldots,x_m]$$ is a denormalized volume polynomial, then the diagonalization $p(v_1,\ldots,v_{k},u,\ldots,u)$ is also a denormalized volume polynomial.
\end{theorem}

Using this we can prove that a number of other operations preserve being a denormalized volume polynomial:

\begin{theorem}[Products, =Corollary~\ref{cor:denorm-product}]
The product of two denormalized volume polynomials is again a denormalized volume polynomial.
\end{theorem}

\begin{theorem}[Normalization, $\subset$ Corollary~\ref{cor:normalizationpreservesvolume}]
If $\sum_\alpha c_\alpha x^\alpha$ is a denormalized volume polynomial then its normalization
$$\sum_\alpha \frac{c_\alpha}{\alpha!} x^{\alpha}$$
is again a denormalized volume polynomial.
\end{theorem}

\begin{theorem}[Truncation, =Corollary~\ref{cor:truncation}]
If $\sum_\alpha c_\alpha x^\alpha$ is a denormalized volume polynomial and $\beta$ is given then the truncation
$$\sum_{\alpha\ge \beta}  c_{\alpha} x^{\alpha}$$
is again a denormalized volume polynomial.
\end{theorem}

By an analogous result for denormalized Lorentzian polynomials \cite[Lemma~4.8.]{branden2021lower} we know that diagonalizations of denormalized volume polynomials are denormalized Lorentzian. However, this would be a strictly weaker statement, see \cite[Example~14]{zbMATH07822675}. In particular, only a weaker form of the reverse Khovanskii-Teissier inequality (\ref{eq:rKT}) can be deduced, see \cite{hu2023intersection}.

In Section~\ref{sec:posets} we give an illustrating application of Theorem~\ref{thm:diagonalization} in the context of poset inequalities.

\begin{remark}
Since being a volume polynomial is stronger than being Lorentzian, our results along with  Theorem~\ref{thm:diagonalization} in conjunction with (\ref{eq:rKT}) (or (\ref{eq:khovanskii-teissier})) gives rise to inequalities for intersection numbers (which are new, at least as far as we know).  For instance if $A_i$ and $B_i$ are ample divisors on $X$ then
  \begin{align*}
    &\left(\sum_{i+j=d} \left(A_1^i \cdot A_2^j\right)\right) \left(B_1\cdots B_d\right)\\
    \leq& \binom{d}{k} \left(\sum_{i+j=k} \left(A_1^i\cdot A_2^j \cdot B_{k+1}\cdots B_d\right) \right) \left(\sum_{i+j=d-k} \left(A_1^i\cdot A_2^j \cdot B_{1}\cdots B_k\right) \right)  
  \end{align*}
  holds.
\end{remark}

\subsection*{Acknowledgements} We would like to thank Thomas Wannerer for invaluable discussions, which initiated the whole project. This material is based upon work supported by the National Science Foundation under Grant No. DMS-1749447. The second named author was supported by the EPSRC grants EP/V013270/1, EP/V055445/1 and by funding from the Carl Zeiss Foundation.



\section{Preliminaries}\label{sec:preliminaries}
\subsection{Notation and Conventions}
In this paper we denote the non-negative integers by $\NN$.  For an elements $\alpha = (\alpha_1,\ldots,\alpha_n) \in  \NN^n$ and $\beta = (\beta_1,\ldots,\beta_n) \in  \NN^n$ set $|\alpha|=\sum_i \alpha_i$, $\alpha!=\prod_i \alpha_i!$, $\binom{\alpha}{\beta}=\prod_i \binom{\alpha_i}{\beta_i}$ and $\binom{|\alpha|}{\alpha}=\frac{|\alpha|!}{\alpha!}$. 

We denote the $i$-th canonical basis vectors in $\NN^^n$ by $e_i$ and the element $(d,\ldots,d) \in \NN^n$ by $\underline{\mathbf{d}}$.  We write $x^\alpha$ for the monomial in $\RR[x_1,\ldots,x_n]$ with exponent vector $\alpha$, i.e. $x^\alpha = \prod_{i=1}^n x_i^{\alpha_i}$. Let $s \in \RR_{\kappa}[x_1, \ldots x_n]$ be a homogeneous polynomial. Then we also consider the corresponding differential operator $\partial_s:=s(\partial_{x_1},\ldots,\partial_{x_n}) \in \RR[\partial_{x_1},\ldots,\partial_{x_n}]$.

Moreover, for any homogeneous polynomial $s(x_1,\ldots,x_n)$ we define the $j$-th \emph{derived polynomial} $s^{(j)}(x_1,\ldots,x_n)$ by requiring $$s(x_1+u,\ldots,x_n+u) = \sum_{j=0}^d s^{(j)}(x_1,\ldots,x_n) u^j.$$  
\subsection{Volume polynomials and denormalized volume polynomials}
\label{sec:tools}
\begin{definition}
  Given a  projective variety $X$ of dimension $d$ and nef  divisor classes $D_1,\ldots,D_n$ we call
  \[\left(\sum x_i D_i \right)^d\]
  the \emph{associated volume polynomial}. More generally, we also call  limits of positive multiples of such polynomials \emph{volume polynomials}.

  The \emph{normalization} of a homogeneous polynomial $f=\sum_\alpha c_\alpha x^\alpha$ is defined by
  \[
    N(f)=\sum_\alpha \frac{c_\alpha}{\alpha!}  x^\alpha.
  \]
  We call $N^{-1}(f)$ the \emph{denormalization} of $f$.
  
  Consequently we say that $f$ is a \emph{denormalized volume polynomial} if $N(f)$ is a volume polynomial.
\end{definition}


\begin{lemma}
  \label{lem:volume-trafo}
  Assume $f$ is a volume polynomial and $A$ is a matrix with non-negative entries. Then $f(Ax')$ is again a volume polynomial.
\end{lemma}
\begin{proof}
  Assume $f$ is the volume polynomial associated to  $X,D_1,\ldots D_n$. Then $f(Ax')$ is the volume polynomial associated to $X$ and
  $\sum_i A_{i1}D_i, \ldots, \sum_i A_{im}D_i$.
\end{proof}

\begin{remark}
  \label{rem:volume-diagonalization}
  Lemma~\ref{lem:volume-trafo} shows in particular that if $f(x_1,\ldots,x_n)$ is a volume polynomial, then  the diagonalization \(g(x_1,\ldots,x_{n-1})=f(x_1,\ldots,x_{n-1},x_{n-1})\) is also a volume polynomial. As we will see, to obtain the same result for denormalized volume polynomials requires are more involved argument.
\end{remark}

\begin{lemma}
  \label{lem:volume-product}
  Products of volume polynomials are again volume polynomials. 
\end{lemma}
\begin{proof}
Consider volume polynomials $f(x_1,\ldots, x_n)$, $f'(y_1,\ldots,y_m)$ associated to $X$, $D_1,\ldots, D_n$  and $X'$, $D'_1,\ldots,D'_m$, respectively. Assume first that sets of variables are disjoint.  Then $f\cdot f'$ is the volume polynomial associated to $X \times X'$ and $\pi^*D_1,\ldots,\pi^*D_n,(\pi')^*D'_1,\ldots,(\pi')^*D_m'$, where $\pi$ and $\pi'$ are the projections to $X$ and $X'$, respectively. The general case with non disjoint variables then follows from Remark~\ref{rem:volume-diagonalization}.
\end{proof}

\begin{lemma}
  \label{lem:volume-monomial}
  Every monomial is a volume polynomial.
\end{lemma}
\begin{proof}
  The monomial $x_1^d$ is the volume polynomial associated to $\PP^d$ and a hyperplane $H \subset \PP^d$. For general monomials the claim follows by Lemma~\ref{lem:volume-product}
\end{proof}

\begin{proposition}
  \label{prop:volume-schubert}
  Let $s \in \RR[x_1,\ldots,x_n]$ be a Schubert polynomial and $f$ a volume polynomial in $\in \RR[x_1,\ldots,x_n]$. Then $\partial_s(f)$ is again a volume polynomial.
\end{proposition}
\begin{proof}
  Assume $f$ is the volume polynomial corresponding to the divisor classes $D_1, \ldots, D_n$ on $X$. By \cite[Thm.~8.2]{zbMATH00062989} there is an irreducible subvariety $Y$
  such that for the class $[Y]$ in the Chow ring of $X$ we have $[Y]=s(D_1,\ldots,D_n)$. In particular, we obtain
  \[
    \left(\sum x_i (D_i|_Y) \right)^{d-\dim(Y)}= s(D_1,\ldots,D_n) \cdot \left(\sum x_i D_i \right)^{d-\dim(Y)}=\partial_s(f).
  \]
Hence, $\partial_s(f)$ is the volume polynomial associated to $Y$ and $D_1|_Y, \ldots, D_n|_Y$. 
\end{proof}

\begin{remark}
  \label{rem:volume-partial-derivatives}
  Since $x_i$ is a Schubert polynomial, Proposition~\ref{prop:volume-schubert} also implies that partial derivatives of volume polynomials are again volume polynomials.
\end{remark}

\begin{lemma}
  \label{lem:rKT}
  Let $f(x_1,\ldots,x_n)=\sum_{|\tau|=d} c_\tau x^\tau$ be a denormalised volume polynomial. Then for any $\alpha,\beta,\gamma$ with $|\alpha| + |\beta| + |\gamma|=d$ we have  
   \begin{equation}  \label{eq:rKT}c_{(\alpha+ \beta+\gamma)} \cdot c_{(\alpha+ |\beta+\gamma|e_1)} \leq \binom{|\beta+\gamma|}{|\beta|} c_{(\alpha + |\gamma|e_1+\beta)}\cdot c_{(\alpha + |\beta|e_1+ \gamma)}\end{equation}
 \end{lemma}
 \begin{proof}
   By assumption $N(f)$ and by Remark~\ref{rem:volume-partial-derivatives} also $g=\partial^\alpha N(f)$ are volume polynomials.   We write the latter in the form $(x_1A + \sum_{i\geq 2} x_i B_i)^{|\beta+\gamma|}$. Then we obtain
   \[N^{-1}(g)=\sum_\lambda c'_\lambda x^\lambda=\sum_\lambda c_{(\alpha+\lambda)} x^\lambda\]
   with $c'_\lambda$ being equal to the intersection number $(A^{\lambda_1}\cdot B_2^{\lambda_2}\cdots B_n^{\lambda_n})$. Theorem~3.5 in \cite{zbMATH07745002} then takes the form
   \[
     c'_{|\beta+\gamma|e_1}c'_{\beta+\gamma} \leq \binom{|\beta+\gamma|}{|\beta|}c'_{(|\beta|e_1+\gamma)}c'_{(|\gamma|e_1+\beta)}.
   \]
  which is precisely \eqref{eq:rKT}.
 \end{proof}



 We finish this section with a few simple corollaries of our main theorem providung us with more operators which preserve the property of being a (denormalized) volume polynomial.
 \begin{corollary}
  \label{cor:denorm-product}
  The product of two denormalized volume polynomials is again a denormalized volume polynomial.
\end{corollary}
\begin{proof}
  We first consider denormalized volume polynomials $f$ und $g$ with a disjoint set of variables.
  Then Lemma~\ref{lem:volume-product} implies that $N(f)N(g)$ is a volume polynomial. As the variables are disjoint, $N^{-1}(N(f)N(g))=fg$ and so $fg$ is a denormalized volume polynomial.
  The general case without disjoint variables follows from Theorem~\ref{thm:diagonalization}.
\end{proof}

\begin{corollary}
  \label{cor:vol-antiderivative}
  Given a volume polynomial $f=\sum_\alpha c_\alpha x^\alpha$  the antiderivative $$\int^\gamma f :=\sum_\alpha \frac{\alpha!}{(\alpha+\gamma)!}c_\alpha x^{\alpha+\gamma}$$ is again a volume polynomial.
\end{corollary}
\begin{proof}
  We have $\int^\gamma f=N(x^\gamma N^{-1}(f))$ and the claim follows from Corollary~\ref{cor:denorm-product}.
\end{proof}

For an exponent vector $\gamma$ we consider the lower truncation operator which sends $f=\sum_\alpha c_\alpha x^\alpha$ to $f_{\geqslant \gamma}:= \sum_{\alpha \geq \gamma} c_\alpha x^\alpha$.
\begin{corollary}\label{cor:truncation}
  For a (denormalized) volume polynomial $f$ the lower truncation $f_{\geqslant \gamma}$ is again a (denormalized) volume polynomial.
\end{corollary}
\begin{proof}
  Assume that $f$ is a volume polynomial. Then $f_{\geqslant \gamma}= \int^\gamma \partial^\gamma f$ is also a volume polynomial by Corollary~\ref{cor:vol-antiderivative}. On the other hand, we have $f_{\geqslant \gamma}= N^{-1}(N(f)_{\geqslant \gamma})$. Hence, the corresponding result for denormalized volume polynomials follows, as well.
\end{proof}
 
\section{Proof of the Main Theorem}
Here we apply the tools from the previous section to show the diagonalizations of denormalized volume polynomials are again denormalized volume polynomials.

\begin{lemma}
\label{lem:technical-derived}
  Consider a homogeneneous polynomial $p \in \RR[v_1,\ldots,v_{k},x_1,\ldots,x_m]$ of degree $d$ and polynomials $p_\iota \in \RR[x_1,\ldots,x_m]$ with $p=\sum_\iota p_\iota(x_1,\ldots,x_m)v^\iota$.  For  any homogeneous polynomial $s \in \RR[x_1,\ldots,x_m]$ of degree $a$ the identity   \[\partial_s|_{x=0} \bigg(p \cdot \Big(u + \sum_{i=1}^m x_i\Big)^{a}\bigg)=\sum_{|\iota|\leq d} \frac{a!}{(d-|\iota|)!}\partial_{s^{(|\iota|+a-d)}}p_\iota \cdot v^\iota u^{d-|\iota|}\]  
  holds.
\end{lemma}

Note, that here the term $\partial_{s^{(|\iota|+a-d)}}(p_\iota)$ has degree $0$ and is therefore just a (non-negative) real number.

\begin{proof}
  It is sufficient to prove the statement for monomials $s=x^\alpha$ and $p=x^\beta v^\iota$, as the general case follows by linearity.
  First consider
  \begin{equation*}
  (x_1+u)^{\alpha_1}\cdots (x_m+u)^{\alpha_m}
    \sum_{|\beta|+i=|\alpha|}\binom{\alpha}{\beta} x^\beta u^i.
\end{equation*}
By definition, $s^{(|\iota|+a-d)} =  \sum_{|\beta|=d-|\iota|}\binom{\alpha}{\beta} x^\beta$ and therefore
\begin{equation}
\partial_{s^{(|\iota|+a-d)}}(p_\iota) =\partial_{s^{(|\iota|+a-d)}}(x^\beta)=\beta!\binom{\alpha}{\beta}.
    \label{eq:derived-of-monomial} 
\end{equation}
 As usual, we consider $\binom{\alpha}{\beta}=0$ if it is not the case that $\beta\leq \alpha$.

  On the other hand, we obtain
  \begin{align*}
    \partial_s|_{x=0} \Big(p \cdot \big(u + \sum_{i=1}^m  x_i \big)^{a}\Big) &=
    \partial^\alpha|_{x=0} \left(x^\beta v^\iota\cdot \left( u+ \sum_{i=1}^{m}  x_i \right)^{a}\right) \\
                                                                                                   &=     \partial^\alpha|_{x=0} \left(x^\beta v^\iota \sum_{j=0}^{a} \binom{a}{j}\left(\sum_{i=1}^m x_i\right)^{a-j}u^j\right)\\
                                                                   &=\partial^\alpha|_{x=0} \left(x^\beta v^\iota \sum_{j=0}^{a} \binom{a}{j}\left(\sum_{|\gamma|=a-j} \binom{a-j}{\gamma_1,\ldots,\gamma_m} x^\gamma \right)u^j\right)\\
                                                                   &=\partial^\alpha|_{x=0}\left(\binom{a}{d-|\iota|} \binom{a-d+|\iota|}{ \alpha_1-\beta_1,\ldots,\alpha_m-\beta_m} x^{\alpha}\cdot v^\iota u^{d-|\iota|}\right)\\
                                                                   &=\alpha!\binom{a}{d-|\iota|} \binom{a-d+|\iota|}{ \alpha_1-\beta_1,\ldots,\alpha_m-\beta_m} \cdot v^\iota u^{d-|\iota|}\\
                                                                   &=\frac{a!}{(d-|\iota|)!}\cdot \beta! \binom{\alpha}{\beta}\cdot v^\iota u^{d-|\iota|}.
  \end{align*}
where the fourth equality uses that all terms vanish unless $\beta+ \gamma=\alpha$ which forces $|\gamma|=a-j = |\alpha|-|\beta|$ so $j = |\beta| = d-|\iota|$. Comparing with (\ref{eq:derived-of-monomial}) shows the claim.
\end{proof}

Given a polynomial $f=\sum_\beta c_\beta x^\beta$ in $\RR[x_1,\ldots,x_m]$
and an exponent vector $\alpha \in \NN^m$, we define the \emph{upper truncation} $$f_{\leqslant \alpha} =\sum_{\beta\le \alpha} c_\beta x^\beta.$$ Then the truncation $f_{\leqslant \underline{\mathbf{1}}}$ is called the multiaffine part of $f$. The following is an analog of \cite[Corollary 3.5]{BrandenHuh} for (denormalized) volume polynomials.
\begin{proposition}
  \label{prop:weighted-truncation}
  For $f=\sum_\beta c_\beta v^\beta  \in \RR[v_1,\ldots,v_n]$ a (denormalized) volume polynomial and $\alpha \in \NN^n$ the polynomial $\sum_{\beta}\binom{\alpha}{\beta} c_\beta v^\beta$ is again a (denormalized) volume polynomial. In particular, the multiaffine part $f_{\leqslant \underline{\mathbf{1}}}$ is again a (denormalized) volume polynomial.
\end{proposition}
\begin{proof}
  By repeatidly applying the statement, it is enough to prove that for $a \in \NN$ the polynomial $\sum_{\beta}\binom{a}{\beta_n} c_\beta v^\beta$ is again a (denormalized) volume polynomial.  We apply Lemma~\ref{lem:technical-derived}  to $p=f$ for $k=n-1$, $m=1$, $x_1=v_n$ and $s=v^{a}_n$.  So with this notation $p_\iota = p_\iota(v_n) = c_{\iota,b} v_n^{d-|\iota|}$.   Next note that
  \[\partial_{s^{(|\iota|+a-d)}}(v^{d-|\iota|}_n)=(d-|\iota|)!\binom{a}{d-|\iota|}\]
  and with our notation Lemma~\ref{lem:technical-derived}  gives
  \begin{equation}
    \label{eq:weighted-truncation}
    \frac{1}{a!}\cdot \partial_s|_{v_n=0} \left(p \cdot (u + v_{n})^{a}\right)=
    \sum_{|\iota|+b=d} \binom{a}{b} c_{\iota,b} \cdot v^\iota u^b
  \end{equation}
  with the usual convention that $\binom{a}{b} = 0$ for $b>a$. Since the left-hand-side of (\ref{eq:weighted-truncation}) is a volume polynomial by Proposition~\ref{prop:volume-schubert} and Lemma~\ref{lem:volume-trafo},  the claim for volume polynomials follows.

  Since the considered operators commute with normalization, we also get the denormalized version of the result.
\end{proof}

As a consequence we obtain the following analog of \cite[Corollary~3.7]{BrandenHuh} for (denormalized) volume polynomials.
\begin{corollary}\label{cor:normalizationpreservesvolume}
  The normalization $N(f)$ of a (denormalized) volume polynomial is again a (denormalized) volume polynomial. 
\end{corollary}
\begin{proof} Consider $f=\sum_\beta c_\beta x^\beta  \in \RR[x_1,\ldots,x_n]$.  We apply Proposition~\ref{prop:weighted-truncation} to the polynomial $a^{-d}\cdot f(x_1, \ldots,x_n)$ and $\alpha=\underline{\textbf{a}}$. Then we conclude that
  \[N(f)=\lim_{a\to \infty} \sum_{\beta} \binom{\underline{\mathbf{a}}}{\beta} a^{-d} c_\beta x^\beta\]
  is again a volume polynomial.
\end{proof}


\begin{proposition}
  \label{prop:general-kahn-saks}
  For $\ell >0$ and $p = \sum_\iota p_\iota v^\iota \in \RR[v_1,\ldots,v_{k}, x_j \mid j \geq 1]$ being a volume polynomial, the polynomial
  \begin{equation}
      q(v_1,\ldots,v_{k},u)=\sum_{|\iota| \leq d} \frac{N^{-1}(p_\iota)(u,\ldots,u)}{(d-|\iota|)!}\cdot v^\iota \label{eq:volume-poly}
  \end{equation}
is a volume polynomial, as well.
\end{proposition}
\begin{proof}
Let $s_d:=s_{d,m}:= \sum_{|\delta|=d}x^\delta$ be the complete homogeneous symmetric polynomial of degree $d$ in the variables $x_1,\ldots,x_m$.
What is key in this proof is that $s_{d}$ is a Schubert polynomial and for every homogeneous polynomial $f  \in \RR[x_1,\ldots,x_m]$ of degree $d$ we have $\partial_{s_{d}}f=N^{-1}(f)(1,\ldots,1)$, or said another way
  $$(\partial_{s_{d}}f) u^d = N^{-1}(f) (u,\ldots,u).$$
We observe also the identity
  \begin{equation}
    s_{d}^{(i)}=\binom{d+m-1}{i}s_{d-i}.\label{eq:derived-of-complete-homogeneous}
  \end{equation}
  By Lemma~\ref{lem:volume-monomial} and Lemma~\ref{lem:volume-trafo} the polynomial $p \cdot \Big(u+\sum_{i=1}^m x_i \Big)^{d+\ell}$ is a volume polynomial for every $\ell$ and $m$. It then follows from Proposition~\ref{prop:volume-schubert}
  that also 
  \begin{equation}
    \label{eq:qell-in-disguise}
    \partial_{s_{d+\ell}}|_{x=0} \bigg(p \cdot \Big(u+\sum_{i=1}^m x_i\Big)^{d+\ell}\bigg)
  \end{equation}
  is a volume polynomial, since $s_{d+\ell}$ is a Schubert polynomial. Now Lemma~\ref{lem:technical-derived} and (\ref{eq:derived-of-complete-homogeneous}) imply that (\ref{eq:qell-in-disguise}) coincides with
    \begin{equation}
      q_{\ell,m}(v_1,\ldots,v_{k},u)=\sum_{|\iota| \leq d} \frac{(d+\ell)!}{(d-|\iota|)!} \binom{d+\ell+m-1}{|\iota|+\ell}
N^{-1}(p_\iota)(u,\ldots,u)
      \cdot v^\iota \label{eq:qell}
  \end{equation}

  Hence, by Lemma~\ref{lem:volume-trafo} also the limit $\lim_{\ell \to \infty}\frac{1}{\ell^{(m-1)}(d+\ell)!}q_{\ell,m}(v_1\ldots,v_{k},\ell^{-1}u)$ is a volume polynomial (if the limit exists). Here, we obtain 

 \begin{equation}
 \frac{q_{\ell,m}\left(v_1,\ldots,v_{k},\ell^{-1}u\right)}{\ell^{m-1} \cdot (d+\ell)!}=\sum_{\iota}\frac{(\ell+|\iota|+1) \cdots (\ell+d+m-1)}{(d+m-|\iota|-1)! \ell^{(d-|\iota|+m-1)}}\cdot \frac{N^{-1}(p_\iota)(u,\ldots,u)}{(d-|\iota|)!} \cdot v^\iota.\label{eq:limit-poly2}  
\end{equation}

Now, we observe that the expression in~(\ref{eq:limit-poly2}) converges to 
     \[
     q_m(v_1,\ldots,v_{k},u)=\sum_{\iota} \frac{N^{-1}(p_\iota)(u,\ldots,u)}{(d-|\iota|)!(d+m-|\iota|-1)!} \cdot v^\iota   \]
  as $\ell \to \infty$. Finally, we observe that
  \[q(v_1,\ldots,v_k,u)=\lim_{m\to \infty}(m-1)!\cdot q_m (v_1,\ldots,v_{k},m \cdot u).\]
\end{proof}

  
\begin{proof}[Proof of Theorem~\ref{thm:diagonalization}]
  If $p(v_1,\ldots,v_k,x_1,\ldots,x_m)=\sum_\iota p_\iota(x_1,\ldots,x_m)v^\iota$ is a denormalized volume polynomial, then $N(p)$ is a volume polynomial by definition. Now, we apply  Proposition~\ref{prop:general-kahn-saks}  to $N(p)$ and obtain that
  \begin{align*}
    \sum_{|\iota| \leq d} \frac{N^{-1}(N(p_\iota))(u,\ldots,u)}{(d-|\iota|)!}\frac{v^\iota}{\iota !}&=
                                                                                                            \sum_{|\iota| \leq d} \frac{p_\iota(u,\ldots,u)}{(d-|\iota|)!} \frac{v^\iota}{\iota !}
    &=N(p(v_1,\ldots,v_{k},u,\ldots,u)).
  \end{align*}
  is a volume polynomial.
\end{proof}


\begin{remark}
  Our proof actually applies to (denormalizations of) all classes of polynomials, which are preserved under products with linear forms and differential operators with constant coefficients associated to complete homogeneous symmetric polynomials. In particular it applies to (denormalized) Lorentzian polynomial (a simpler proof of this statement about Lorentzian polynomials be found in \cite[Lemma~4.8.]{branden2021lower}).
\end{remark}

\begin{remark}
  Note that all the tools from Section~\ref{sec:tools} are based on geometric constructions. Hence, our proof has a geometric interpretation, as well. Given a volume polynomial coming from divisor classes on some variety $X$  our construction  produces a sequence of volume polynomials living on subvarieties  in  products of $X$ with projective spaces.
\end{remark}

\goodbreak
\section{Application to poset inequalities}
\label{sec:posets}

Let $P$ be a poset of size $n$.   We let $\mathcal E(P)$ be the set of linear extensions of $P$, i.e.bijections $L:P\to \{1,\ldots,n\}$ such that if $x<y$ then $L(x)<L(y)$.

\begin{definition}
  For a poset $P$ of size $n$ and fix a chain $x_1<x_2<\cdots <x_k$ in $P$.  We set
  \(
    F(\iota)=F(\iota_1\ldots,\iota_{k-1})=\#\{L \in \mathcal{E}(P) \mid L(x_{m+1})-L(x_{m})=\iota_m\}
  \)
  and define the \emph{Kahn-Saks polynomial}
  \[ P_{KS} (u,v_1,\ldots,v_{k-1}) := \sum_{|\iota|+j = n-k} {F(\iota_1+1,\ldots,\iota_{k-1}+1)}v^\iota u^{j}.\]
\end{definition}

As an illustration of our main theorem we will show that the Kahn-Saks polynomial is a a denormalized volume polynomial. Note however, that a more direct proof (stated explicitly for  the case $k=3$) can be found in \cite{Chan_Pak_Panova_2024}.

\medskip

For a poset $P=\{x_1,\ldots,x_k,y_{1},\ldots,y_{n-k}\}$ we first consider 
\(\mathcal N(i_1,\ldots,i_k):=|\{L \in \mathcal{E}(P) \mid \forall_j \colon L(x_j)=i_j\}| \)
and the polynomial
\begin{equation}
  p(v_1,\ldots,v_{k-1},u_1,u_2):=\!\!\!\sum_{a+b+|\iota|=n-k} \!\!\!\mathcal N\left(a+1,a+\iota_1,\cdots,a+\sum_j\iota_j +k\right) u_1^au_2^bv^\iota .\label{eq:p}
\end{equation}
Stanley in \cite[Theorem 3.2]{zbMATH03760173} proves that this is a denormalized volume polynomial of in the sense of convex geometry. The theory of toric varieties shows that this can be interpreted also as a volume polynomial in our sense, see \cite[Section~5.4]{zbMATH00447302}.

From this we conclude
  \begin{theorem}
  \label{thm:kahn-saks-polynomial}
  The Kahn-Saks polynomial $P_{KS}$ is a denormalized volume polynomial.
\end{theorem}
  \begin{proof}
    Taking into account that the polynomial $p$ from \eqref{eq:p} is a denormalized volume polynomial, Theorem~\ref{thm:diagonalization} applies to give that
   \[
     \sum_{|\iota| \leq n-k}\; {\sum_{a\leq n-k-|\iota|}\mathcal N\left(a+1,a+\iota_1,\cdots,a+\sum_j\iota_j +k\right)}v^\iota u^{n-k-|\iota|}
   \]
   is a denormalized volume polynomial.   Now, the result follows from the observation that
    \(\sum_{a\leq n-k-|\iota|}\mathcal N\left(a+1,a+\iota_1,\cdots,a+\sum_j\iota_j +k\right) = F(\iota_1+1,\ldots,\iota_{k-1}+1).\)
  \end{proof}

\begin{corollary}
  \label{cor:AF-type}
   For any  chain $x_1<x_2<\ldots<x_k$ of a poset $P$ and $\iota \in \NN^{k-1}$ we have
  \[
    F(\iota)^2 \geq F(\iota-e_i)F(\iota+e_i) \text{ and } F(\iota)^2 \geq F(\iota-e_i+e_j)F(\iota+e_i-e_j)
  \]
\end{corollary}
\begin{proof}
  This follows directly from the Khovanskii-Teissier inequality (\ref{eq:khovanskii-teissier}) for its coefficients.
\end{proof}

\begin{corollary}[Kahn-Saks Inequality,  \cite{zbMATH03893257}]
For a chain $x_1<x_2$ in $P$ we have $$F(i)^2 \geq F(i-1)F(i+1).$$
\end{corollary}
\begin{proof} This is the $k=2$ case of the previous corollary.
\end{proof}

\begin{corollary}
  \label{cor:weak-generalized-crossproduct}
  We have
  \[
   F_P(\alpha)F_P(\alpha + \beta + \gamma) \leq   \binom{|\beta|+|\gamma|}{|\beta|}F_P(\alpha+ \beta)F_P(\alpha+ \gamma)
  \]
\end{corollary}
\begin{proof}
   This follows directly from the reverse Khovanskii-Teissier inequality (\ref{eq:rKT}) for its coefficients.
\end{proof}

\begin{remark}
  The fact that the Kahn-Saks polynomial $\PKS$ is denormalized Lorentzian only gives the  weaker estimate
  \[
    F_P(\alpha)F_P(\alpha + \beta + \gamma) \leq   2^{|\beta||\gamma|}F_P(\alpha+ \beta)F_P(\alpha+ \gamma),
  \]
  see \cite[Thm.~2.5]{hu2023intersection}.
\end{remark}

Corollary~\ref{cor:weak-generalized-crossproduct} can be interpreted as a weak form of the generalized cross-product conjecture from \cite{zbMATH07576885}, which states that
\begin{equation}
  F_P(\alpha)F_P(\alpha + \beta + \gamma) \leq   F_P(\alpha+ \beta)F_P(\alpha+ \gamma)\label{eq:generalized-crossproduct}    .
\end{equation}

Let us formulate an even stronger version of this conjecture.
\begin{conjecture}
  \label{con:kahn-saks-rayleigh}
  The normalized Kahn-Saks polynomial for a chain $x_1 \leq  \ldots \leq x_k$ is $1$-Rayleigh, i.e. 
  \[
    \partial^\alpha  N(\PKS) \partial^{\alpha +e_i +e_j} N(\PKS)(u,v)  \leq   \partial^{\alpha+e_i}  N(\PKS) \partial^{\alpha + e_j} N(\PKS)(u,v)
  \]
  or every $(u,v) \in \RR_{\geq 0}^k$.
\end{conjecture}

\begin{proposition}
  \label{prop:rayleigh-crossproduct}
  Conjecture~\ref{con:kahn-saks-rayleigh} implies the  generalized cross-product conjecture (\ref{eq:generalized-crossproduct})
  for chains $x_1 \leq  \ldots \leq x_k$ in a poset $P$.
\end{proposition}

To show this we first need to consider ordinal sums of our posets. Let $P,Q$ be posets. Then the ordinal sum $P \oplus Q$ has the disjoint union of $P$ and $Q$ is its underlying set, with partial order given by \[x \leq_{P \oplus Q} y \qquad \Leftrightarrow \qquad x \leq_P y \text{ or } x \leq_Q y \text{ or } (x \in P \text{ and } y \in Q).\]
\begin{lemma}
  \label{lem:poset-augmentation}
  Consider $P^{(\ell)}=P \oplus [\ell]$ for the totally order set $[\ell]=\{z_1\leq  \ldots \leq z_\ell\}$.
  Given $x_1,\ldots,x_k \in P \subset P'$ and $\iota=(\iota_1, \ldots, \iota_{k-1}) \in \ZZ^{k-1}$ we then
  have \[F_{P^{(\ell)}}(\iota)=F_P(\iota).\]
\end{lemma}
\begin{proof}
  Assume $|P|=n$ and let $L \in \mathcal{E}(P^{(\ell)})$ be a linear extension, then we must have $L(z_i)=n+i$,
  since we have $z_1\leq  \ldots \leq z_\ell$  and those elements are larger than every element in $P$.
  Hence, $L \mapsto L|_{P}$ induces a bijection $\mathcal{E}(P^{(\ell)}) \to \mathcal{E}(P)$, which obviously preserves the distances $L(x_{i+1})-L(x_i)$.
\end{proof}

\begin{proof}[Proof of Proposition~\ref{prop:rayleigh-crossproduct}]
  Given a poset of size $n$ and a chain of size $k$.  By \cite[(14)]{hu2023intersection} the $1$-Rayleigh property implies  \[\partial^\alpha  N(\PKS) \partial^{\alpha+\beta+\gamma}N(\PKS) (u,v) \leq \partial^{\alpha+\beta} N(\PKS) \partial^{\alpha+\gamma}N(\PKS) (u,v)\] for every $(u,v) \in \RR_{\geq 0}^k$. Evaluating in $(1,\underline{\mathbf 0})$ gives
  \[
    \frac{F_P(\alpha)}{(n-k-|\alpha|)!}\frac{F_P(\alpha+ \beta+\gamma)}{(n-k-|\alpha|-|\beta|-|\gamma|)!} \leq \frac{F_P(\alpha+\beta)}{(n-k-|\alpha|-|\beta|)!}\frac{F_P(\alpha+\gamma)}{(n-k-|\alpha|-|\gamma|)!} 
  \]
  and therefore
  \begin{align*}
    F_P(\alpha)F_P(\alpha+\beta+\gamma) &\leq \frac{(n-k-|\alpha|)!(n-k-|\alpha|-|\beta|-|\gamma|)!}{{(n-k-|\alpha|-|\beta|)!(n-k-|\alpha|-|\gamma|)!}}F_P(\alpha+\beta)F_P(\alpha+\gamma) \\
                                              &= \frac{\prod_{i=1}^{|\beta|}(n-k-|\alpha|-|\beta|+i)}{\prod_{i=1}^{|\beta|}(n-k-|\alpha|-|\beta|-|\gamma|+i)}F_P(\alpha+\beta)F_P(\alpha+\gamma).
  \end{align*}
  Applying the above to the poset $P^{(\ell)}=P \oplus [\ell]$ of size $n+\ell$ and taking into account Lemma~\ref{lem:poset-augmentation} we obtain
  \[ F_P(\alpha)F_P(\alpha+\beta+\gamma)  \leq   \prod_{i=1}^{|\beta|} \frac{n+\ell-k-|\alpha|-|\beta|+i}{n+\ell -k-|\alpha|-|\beta|-|\gamma|+i}F_P(\alpha+\beta)F_P(\alpha+\gamma).\]
  Taking the limit for $\ell \to \infty$ gives the desired result.
\end{proof}

\printbibliography
\end{document}